\documentclass[11pt]{article}
\usepackage{amsfonts}
\usepackage[utf8]{inputenc}
\usepackage{graphics}
\usepackage{cite}
\usepackage{latexsym}
\usepackage{amsmath}
\usepackage{amssymb}
\usepackage[dvips]{epsfig}
\usepackage{amscd}

\newtheorem{theorem}{Theorem}[section]
\newtheorem{remark}{Remark}[section]

\newtheorem{lemma}[theorem]{Lemma}

\newtheorem{proposition}[theorem]{Proposition}

\newcommand{\n}{\rho}

\newcommand{\ti}{\tilde}

\renewcommand{\div}{ {\rm div }  }

\newcommand{\na}{\nabla }

\newcommand{\bt}{\begin{theorem}}
\newcommand{\bl}{\begin{lemma}}
\newcommand{\el}{\end{lemma}}
\newcommand{\et}{\end{theorem}}

\newcommand{\ve}{\varepsilon}
\newcommand{\la}{\label}

\newcommand{\ol}{\overline}

\newcommand{\bn}{\begin{eqnarray}}
\newcommand{\en}{\end{eqnarray}}
\newcommand{\bnn}{\begin{eqnarray*}}
\newcommand{\enn}{\end{eqnarray*}}

\newcommand{\bnnn}{\begin{eqnarray*}}
\newcommand{\ennn}{\end{eqnarray*}}
\newcommand{\ben}{\begin{enumerate}}
\newcommand{\een}{\end{enumerate}}

\newcommand{\ba}{\begin{aligned}}
\newcommand{\ea}{\end{aligned}}
\newcommand{\be}{\begin{equation}}
\newcommand{\ee}{\end{equation}}

\renewcommand{\div}{ {\rm div }  }

\def\norm[#1]#2{\|#2\|_{#1}}

\def\r{\mathbb{R}}

\newcommand{\curl}{{\rm curl} }
\newcommand{\mn}{\mu(\rho)}
\newcommand{\ban}{\bar\rho}

\newcommand{\xmu}{\underline{\mu}}
\newcommand{\smu}{\bar\mu}

\newcommand{\xl}{\left}
\newcommand{\xr}{\right}

\makeatletter
\@addtoreset{equation}{section}
\makeatother

\makeatletter
\@addtoreset{equation}{section}
\makeatother

\title{Global Strong Solutions to Density-Dependent Viscosity Navier-Stokes Equations  in 3D Exterior Domains}

\author{Guocai C{\small AI} $^{a}$,    Boqiang L\"u $^{\rm b}$, Yi Peng $^{\rm c} $ \thanks{ The research of \textsc{B. L\"u} was
partially   supported by NNSFC (No. 11601218) and Natural Science Foundation of Jiangxi Province (No. 20161BAB211002).
 Email: gotry@xmu.edu.cn (G. C. Cai),   ajingli@gmail.com (J. Li), lvbq86@163.com(B. L\"u), pengyi16@mails.ucas.ac(Y.Peng).}\\
{\normalsize a.  School of Mathematical Sciences, }\\ {\normalsize  Xiamen University, Xiamen 361005, P. R. China;}  \\
     {\normalsize   b.  Department of Mathematics \& Institute of Mathematics and } \\{\normalsize Interdisciplinary Sciences,  Nanchang University,} \\ {\normalsize   Nanchang 330031, P. R. China;} \\ {\normalsize   c. College of Mathematics and Physics, }\\ {\normalsize Beijing University of Chemical Technology,} \\ {\normalsize Beijing  100029, P. R. China } }

\date{ }

\begin{document}
\maketitle
\begin{abstract}
The nonhomogeneous Navier-Stokes equations with density-dependent viscosity is studied
in  three-dimensional (3D) exterior domains with  nonslip or slip boundary conditions. We prove that the strong solution exists globally in time provided that the gradient of the initial velocity is suitably small. Here the initial density is allowed to contain vacuum states.  Moreover, after developing some new techniques and methods, the large-time behavior of the strong solutions with exponential decay-in-time rates is also obtained.
\end{abstract}

Keywords: nonhomogeneous Navier-Stokes equations; global strong solution; exponential decay; exterior domain; vacuum.

Math Subject Classification: 35Q30; 76N10.

\section{Introduction}

 We are concerned with the  following nonhomogeneous incompressible  Navier-Stokes equations:
\be\label{1.1}
\begin{cases}
\partial_{t}\rho+\div(\rho u)=0, \\
\partial_{t}(\rho u )+\div(\rho u \otimes u )-\div(2\mn D(u))  +\nabla P=0, \\
\div u=0.
\end{cases}
\ee
Here, $t\ge 0$ is time, $x\in\Omega \subset \r^3$ is the spatial coordinate, and the unknown functions $\rho=\rho(x,t)$, $u=(u^1,u^2,u^3)(x,t)$, and $P=P(x,t)$ are the density, velocity and pressure of the fluid, respectively. The deformation tensor is defined by
\be\la{d}
D(u)=\frac{1}{2}\xl[\na u+(\na u)^T\xr],\ee
and the viscosity   $\mu(\n)$     satisfies  the following hypothesis:
\be\la{n3}
\mu \in C^1[0,\infty),~~   \mu(\n)>0.
\ee
The system reveals a fluid which is given by mixing two miscible fluids that are incompressible and have different densities, or a fluid containing a melted substance (see \cite{L1996}). In this paper, motivated by \cite{HLL2021}, our purpose is to  study the global existence and large-time asymptotic behavior of strong solutions of \eqref{1.1} in the exterior of a simply connected bounded domain in $\r^3$. More precisely, the system \eqref{1.1} will be investigated under the assumptions: The domain $\Omega$ is the exterior of a simply connected bounded smooth domain $D$ in $\r^3$ which represents the obstacle, i.e. $\Omega=\r^3-\bar{D}$ and its boundary $\partial\Omega$ is smooth. The system \eqref{1.1} is equipped with  the given initial data
\begin{equation}\label{1.3}
\rho(x,0)=\rho_0(x),\  \n u (x,0)=m_0(x), \ x\in\Omega.
\end{equation}
In order to close the system,  we need some boundary conditions and a condition at infinity. In this paper, we assume that one of the following two boundary conditions
\be \la{Dir1} u=0 \,\,\,\text{on} \,\,\,\partial\Omega,\ee
or
\be \la{Navi}
u \cdot n = 0, \,\,(D(u)\,n)_{tan}=0 \,\,\,\text{on}\,\,\, \partial\Omega
\ee
holds with the far field behavior
\be \la{inf1} u(x,t)\rightarrow0,\,\,as\,\, |x|\rightarrow\infty. \ee

It is worth noting that \eqref{Dir1} is called no-slip boundary condition and \eqref{Navi} is a kind of slip boundary condition where the symbol $v_{\rm tan }$ represents the projection of tangent plane of the vector $v$ on $\partial\Omega$.

The mathematical study of nonhomogeneous incompressible flow dates back to the late 1970s, which was initiated by the Russian school. When the viscosity coefficient $\mu$ is a positive constant and the density $\rho$ is strictly away from vacuum, Kazhikov obtained the global existence of weak solutions \cite{K1974} and then the local existence of strong ones \cite{AK1973} which  is indeed a global one in either two dimensions or three dimensions with "small" initial data in various certain norms (see \cite{lady, zhang1, zhang4, zhang5, zhang6, dan2} and their references therein). However, the system \eqref{1.1} becomes more complicated when the density is allowed to be vacuum even when $\mu$ is still a positive constant.  The global existence of weak solutions is first established by Simon \cite{S1990}, and under some certain compatibility conditions, Choe-Kim \cite{kim2003} proved the local existence of strong solutions for 3D bounded and unbounded domains. Under some smallness conditions on the initial velocity,  Craig-Huang-Wang \cite{huang13} got a global strong solution to the Cauchy problem ($\Omega=\r^3$).

In the general case that the viscosity coefficient $\mu(\rho)$ depends on the density $\rho$, the global existence of weak solutions is due to Lions \cite{L1996}. Abidi-Zhang  \cite{zhang6}  obtained the global strong solutions strictly away from vacuum when  both $\|\na u_0\|_{L^2}$ and $\|\mu(\n_0)-1\|_{L^\infty}$ are small enough. As for the initial density containing vacuum, Cho-Kim \cite{kim2004} established the existence of the local strong solutions under  compatibility conditions similar to \cite{kim2003} and  Huang-Wang \cite{huang15}, Zhang \cite{zjw} showed  the global strong solutions with small $\|\na u_0\|_{L^2}$ in 3D bounded domains. More recently, He-Li-L\"{u} \cite{HLL2021} obtained the global strong ones to the Cauchy problem with small $\|u_0\|_{\dot{H}^\beta}$ for some $\beta\in(\frac{1}{2}, 1]$ and some extra restrictions on $\mu(\rho)$.

Before stating the main results, we explain the notations and
conventions used throughout this paper. First, for integer numbers $1\le r\le \infty, k\ge 1$, the standard homogeneous and inhomogeneous Sobolev spaces are defined as follows:
   \bnn  \begin{cases}L^r=L^r(\Omega ),\quad
W^{k,r}  = W^{k,r}(\Omega ) , \quad H^k = W^{k,2} ,\\ \|\cdot\|_{B_1\cap B_2}=\|\cdot\|_{B_1 }+\|\cdot\|_{B_2}, \mbox{ for two Banach spaces } B_1 \mbox{ and } B_2, \\
 D^{k,r}=D^{k,r}(\Omega )=\{v\in L^1_{\rm loc}(\Omega )| \na^k v\in L^r(\Omega )\},\\
D^1  =\{v\in L^6 (\Omega )| \na  v\in L^2(\Omega )\},   \\
 C_{0,\div}^\infty=\{f\in C_0^\infty~|~\div f=0\},\quad D_{0,\div}^1=\overline{C_{0,\div}^\infty}~\mbox{closure~in~the~norm~of}~D^{1}.\end{cases}\enn

Next, set
$$ \int fdx\triangleq\int_{\Omega }fdx.$$

Finally, for two $3\times 3$  matrices $A=\{a_{ij}\},\,\,B=\{b_{ij}\}$, the symbol $A\colon  B$ represents the trace of $AB$, that is,
 $$ A\colon  B\triangleq \text{tr} (AB)=\sum\limits_{i,j=1}^{3}a_{ij}b_{ji}.$$

Now, we can give our main results of this paper, which indicate the global existence and large-time behavior of strong solution to the system \eqref{1.1}--\eqref{inf1} in the exterior of a simply connected bounded smooth domain $D$ in $\r^3$.
\begin{theorem}\label{thm1} Let $\Omega$ be the exterior of a simply connected bounded domain $D$ in $\r^3$ and its boundary $\partial\Omega$ is smooth. For constants $\bar\n>0,$ $q\in(3,\infty)$,
assume that the initial data $(\rho_0, m_0)$ satisfy
\begin{equation}\label{2.2}
0\le \rho_{0}\le \bar\n,\  \rho_{0} \in L^{3/2}\cap H^1, \  \na\mu(\n_0)\in L^q  ,\  u_{0}\in D_{0,\div}^1 ,\ m_0=\n_0u_0.
\end{equation}

Then for $$ \xmu\triangleq\min_{\n\in[0,\bar\n]}\mn, \quad\smu\triangleq\max_{\n\in[0,\bar\n]}\mn, \quad M\triangleq\|\na \mu(\n_0)\|_{L^q},$$  there exists some small positive constant $\ve_0$ depending only on $q, \bar\n, \xmu, \smu , \|\n_0\|_{L^{3/2}},$ and $M$ such that if
\begin{equation}\label{xx}
\|\nabla u_0\|_{L^2}\le \ve_0,\end{equation}  the system \eqref{1.1}--\eqref{inf1} admits a unique global strong solution $(\rho, u, P)$ satisfying that
for any $0<\tau<T<\infty$ and $p\in [2,p_0)$  with $p_0 \triangleq\min\{6, q\},$
 \be \label{2.3} \begin{cases}
0\le \rho\in C([0,T]; L^{3/2}\cap H^1 ),
\quad \na\mn \in  C([0,T]; L^q) ,  \\
 \na u  \in L^\infty(0,T;L^2)\cap L^\infty(\tau,T;  W^{1,p_0})\cap C([\tau,T]; H^1\cap W^{1,p })  , \\
 P\in  L^\infty(\tau,T;  W^{1,p_0})\cap C([\tau,T]; H^1\cap W^{1,p }) ,  \\
\sqrt{\rho} u_t\in  L^2(0,T; L^2)\cap L^\infty(\tau,T; L^2),\quad
P_t  \in L^2(\tau,T; L^2\cap L^{p_0}),\\ \na u_t  \in L^\infty(\tau,T; L^2 )\cap    L^2(\tau,T;  L^{p_0}),\quad (\n u_t)_t \in L^2(\tau,T; L^2 ).
\end{cases}\ee
Moreover, it holds that \be \la{oiq1}\sup_{0\le t<\infty} \|\na \n\|_{L^2 }   \le 2 \|\na \n_0\|_{L^2 }  ,\quad \sup_{0\le t<\infty} \|\na \mn\|_{L^q }   \le 2 \|\na \mu(\n_0)\|_{L^q }  , \ee and that there exists some positive constant $\lambda$ depending only on $\|\n_0\|_{L^{3/2}}$ and $\xmu$ such that   for all $t\geq1$,
\be \la{e}
 \|\na u_t(\cdot,t)\|^2_{L^2}+
\|\na  u(\cdot,t)\|_{H^1\cap W^{1,p_0}}^2+\|P(\cdot,t)\|_{H^1\cap W^{1,p_0}}^2\le Ce^{-\lambda t},
\ee
where  $C$ depends only on $q, \bar\n, \|\n_0\|_{L^{3/2}}, \xmu, \smu, M,$ $\|\na u_0\|_{L^2},$ and  $\|\na \n_0\|_{L^2}.$
\end{theorem}


 Furthermore, when the viscosity coefficient $\mu$ is a positive constant, we also have the following conclusion.
\begin{theorem}\label{thm2} Under the conditions of Theorem \ref{thm1} with $\mn\equiv\mu$ for some constant $\mu>0$, there exists some positive constant $\ve$ depending only on $\bar \n$ such that  if $\|\nabla u_0\|_{L^2}\le\mu\ve$, the system \eqref{1.1}--\eqref{inf1} has a unique  global strong solution to  satisfying \eqref{2.3} with $p_0=6$.  In addition, it holds that
 \be \la{oiq2} \sup_{0\le t<\infty}\|\na\n\|_{L^2}\le 2\|\na\n_0\|_{L^2}, \ee and that there is some positive constant   $\lambda$  depending only on $\|\n_0\|_{L^{3/2}}$ and $\mu$ such that  for $t\ge 1,$ \be \la{eq}
 \|\na u_t(\cdot,t)\|^2_{L^2}+
\|\na  u(\cdot,t)\|_{H^1\cap W^{1,6}}^2+\|P(\cdot,t)\|_{H^1\cap W^{1,6}}^2\le Ce^{-\lambda t} ,
\ee where   $C$ depends only on $\bar\n, \mu,$ $\|\n_0\|_{L^{3/2}}$,  $\|\na u_0   \|_{L^2},$ and $\|\na \n_0\|_{L^2}.$

\end{theorem}

\begin{remark} It should be noted here that our  Theorem \ref{thm1} holds
 for any function $\mn$ satisfying \eqref{n3} and for initial density allowed to be arbitrarily large and vacuum under a smallness   assumption  only  on the $L^2$-norm of the gradient of initial velocity.
  \end{remark}

\begin{remark}  Compared with the results of  Guo-Wang-Xie  \cite{GWX1} where they  obtained the global strong axisymmetric solutions for the system \eqref{1.1} in the exterior of a cylinder subject to the Dirichlet boundary conditions,  we only require that the region is any simply connected bounded smooth domain provided    the gradient of initial velocity is suitably small.
\end{remark}

We now make some comments on  the analysis in this paper. The idea mainly comes from the article \cite{HLL2021}. However, because the boundary conditions must be taken into account in our case, we still need some new technical methods to overcome the difficulties caused by them. First, we utilize the cut-off function to combine a priori estimates of the Stokes problem in a bounded domain and the whole space to obtain a priori estimates necessary in the outer region. One can see Lemma \ref{stokes} for details. Next, some exponential decays of the quantities related to density and velocity such as $\|\n^{1/2}u(\cdot,t)\|^2_{L^2}$ and $\|\na u(\cdot,t)\|^2_{L^2}$ are estalbished (see Lemma \ref{lem-ed1}), which are the key to   obtain  the desired  uniform  bound (with respect to time) on the  $L^1(0,T; L^\infty)$-norm
of $\na u$. Finally, based on the a priori estimates we have gotten, we succeed in extending the local  strong solutions whose existence is obtained by Lemma \ref{local} globally in time.

The rest of this paper is organized as follows.  Some  facts and elementary inequalities are collected in Section \ref{sec2}. Section \ref{sec3} is devoted to deriving  a priori estimates of the system \eqref{1.1}. Finally, we will give the proofs of  Theorems  \ref{thm1} and \ref{thm2} in Section 4.

\section{Preliminaries}\label{sec2}
In this sectionn, some facts and elementary inequalities, which will be used frequently later, are collected.

We start with the local existence of strong solutions which has  been proved   in \cite{hls}.
\begin{lemma}\label{local}
Assume that $(\rho_0, u_0)$ satisfies   \eqref{2.2}. Then there exist a time $T_0>0$ and a unique strong solution $(\rho, u, P)$ to the system \eqref{1.1}--\eqref{inf1} in $\Omega\times(0,T_0)$ satisfying \eqref{2.3}.
\end{lemma}
The following lemma  can be found in \cite{fcpm}.
\begin{lemma}
[Gagliardo-Nirenberg]\la{l1} Assume that $\Omega$ is the exterior of a simply connected domain $D$ in $\r^3$. For  $p\in [2,6],\,q\in(1,\infty) $ and
$ r\in  (3,\infty),$ there exists some generic
 constant
$C>0$ which may depend  on $p,q $ and $r$ such that for  any $f\in H^1({\Omega }) $
and $g\in  L^q(\Omega  )\cap D^{1,r}(\Omega ), $
\be\la{g1}\|f\|_{L^p(\Omega )}\le C\|f\|_{L^2}^{\frac{6-p}{2p}}\|\na
f\|_{L^2}^{\frac{3p-6}{2p}},\ee
\be\la{g2}\|g\|_{C\left(\ol{\Omega  }\right)} \le C
\|g\|_{L^q}^{q(r-3)/(3r+q(r-3))}\|\na g\|_{L^r}^{3r/(3r+q(r-3))}.
\ee
\end{lemma}
Generally, \eqref{g1} and \eqref{g2} are called Gagliardo-Nirenberg's inequalities.

The following  Lemmas \ref{crle1}-\ref{crle3} show the control of $\na v$  by means of $\div v$ and $\curl v$ which are very important in our later discussion.
\begin{lemma}  \cite[Theorem 3.2]{vww} \la{crle1}
Let $D$ be a simply connected domain in $\r^3$ with $C^{1,1}$ boundary, and $\Omega$ is the exterior of $D$. For $v\in D^{1,q}(\Omega)$ with $v\cdot n=0$ on $\partial\Omega$, it holds that
\be\la{ljq01}\|\nabla v\|_{L^{q}(\Omega)}\leq C(\|\div v\|_{L^{q}(\Omega)}+\|\curl v\|_{L^{q}(\Omega)})\,\,\,for\,\, any\,\, 1<q<3,\ee
and
$$\|\nabla v\|_{L^q(\Omega)}\leq C(\|\div v\|_{L^q(\Omega)}+\|\curl v\|_{L^q(\Omega)}+\|\nabla v\|_{L^2(\Omega)})\,\,\,for\,\, any\,\, 3\leq q<+\infty.$$
\end{lemma}
\begin{lemma} \cite[Theorem 5.1]{lhm} \la{crle2}
Let $\Omega$ be given in Lemma \ref{crle1}, $1<q<+\infty$, for any $v\in W^{1,q}(\Omega)$ with $v\times n=0$ on $\partial\Omega$, it holds that
$$\|\nabla v\|_{L^q(\Omega)}\leq C(\|v\|_{L^q(\Omega)}+\|\div v\|_{L^q(\Omega)}+\|\curl v\|_{L^q(\Omega)}).$$
\end{lemma}
\begin{lemma} \cite[Lemma 2.9]{CL1}\la{crle3}
Let $D$ be a simply connected domain in $\r^3$ with smooth boundary, and $\Omega$ is the exterior of $D$. For any  $p\in[2,6] $ and    integer $k\geq 0,$  there exists some positive constant $C$ depending only on $p$, $k$ and $D$ such that   every $v\in \{D^{k+1,p}(\Omega)\cap D^{1,2}(\Omega)| v(x,t)\rightarrow 0  \mbox{ as } |x|\rightarrow\infty  \}$  with $v\cdot n|_{\partial\Omega}=0$ or $v\times n|_{\partial\Omega}=0$   satisfies
\be\la{uwkq}\ba\|\nabla v\|_{W^{k,p}(\Omega)}\leq C(\|\div v\|_{W^{k,p}(\Omega)}+\|\curl v\|_{W^{k,p}(\Omega)}+\|\nabla v\|_{L^2(\Omega)}).\ea\ee
\end{lemma}

Considering the Stokes problem
 \be\ba\label{stokes1}
\begin{cases}
-\Delta v  +\nabla \pi=f+\div F, \,\, &x\in\Omega, \\
\div v=\chi\,\, &x\in\Omega.
\end{cases}
\ea\ee
where $\Omega$ is a smooth domain in $\r^3$. Thanks to  \cite{GS1990}, \cite{BM1}, and \cite{AACG2019}, we have the following conclusions about a prior estimates for the problem \eqref{stokes1}.
\begin{lemma}\cite[Theorem 3.1]{GS1990} \la{wse1} Let $\Omega=\r^3$, for the problem \eqref{stokes1} with $f\in W^{-1,q}$ and $F,\chi\in L^q$,  $1<q<+\infty$, there exists a unique distributional solution $v\in W^{1,q}$ such that
\be \la{sswse1}\|\na v\|_{L^q}+ \|\pi\|_{L^q}\le  C(q)  \left(\|f\|_ {W^{-1,q}}+ \|F\|_{L^q}+\|\chi\|_{L^q}\right).\ee
\end{lemma}
\begin{lemma} \cite[Theorem 2.1]{GS1990}\la{eid1} Let $\Omega$ be a  bounded smooth domain in $\r^3$, $f\in W^{-1,q}$ and $F,\chi\in L^q$, $1<q<+\infty$, for the problem \eqref{stokes1} with \eqref{Dir1}, it holds that
\be \la{ssi1}\|\na v\|_{L^q}+ \|\pi-\bar{\pi}\|_{L^q}\le  C(q,\Omega)  \left( \|f\|_ {W^{-1,q}}+\|F\|_{L^q}+\|\chi\|_{L^q}\right).\ee
\end{lemma}
For $1<q<+\infty$, define
\be\ba\label{rq1}
r(q)=
\begin{cases}
\max\{1,\frac{3q}{q+3}\}, \,\, &q\neq \frac{3}{2},\\
>1\,\, &q= \frac{3}{2}.
\end{cases}
\ea\ee
The following lemma give a priori estimates for  the problem \eqref{stokes1} with Dirichlet boundary condition.
\begin{lemma} \la{eed1} Let $\Omega$ be the exterior of a simply connected smooth bounded domain in $\r^3$,  for the problem \eqref{stokes1} with \eqref{Dir1}, \eqref{inf1} and $\chi=0$,  we have the following conclusions:

(1) If $f\in W^{-1,q}$ and $F\in L^q$, $1<q<3$, there exists a unique solution $v\in W^{1,q}$ such that
\be \la{sse1}\|\na v\|_{L^q}+ \|\pi\|_{L^q}\le  C(q,\Omega)  \left(\|f\|_ {W^{-1,q}}+ \|F\|_{L^q}\right).\ee

(2) For any $q\in (\frac{3}{2},+\infty)$, if $F\in L^q$, and $f\in L^{r(q)}$ with $r(q)=\frac{3q}{q+3}$, then
\be \la{sse11}\|\na v\|_{L^q}+ \|\pi\|_{L^q}\le  C(q,\Omega)  \left(\|f\|_ {L^{r(q)}}+ \|F\|_{L^q}+\|\na v\|_{L^2}+ \|\pi\|_{L^2}\right).\ee
.

(3) If $f\in L^q$, $1<q<+\infty$, and $F=0$, then
\be \la{sse2}\|\na^2 v\|_{L^q}+ \|\na \pi\|_{L^q}\le  C(q,\Omega)  \|f\|_ {L^q}.\ee
\end{lemma}
\begin{proof} The first assertion (1) is due to \cite[Theorem 3.5]{BM1} and the last assertion (3) can be found in \cite[Theorem V.4.8]{N1959}.

It remains to prove (2). First, let $B_R\triangleq\{x\in \r^3||x|<R\}$ be a ball whose center is at the origin such that $\bar{D}\subset B_R$. Now we introduce a cut-off function $\eta(x)\in C_c^\infty(B_{2R})$ satisfying $\eta(x)=1$ for $|x|\leq R,$ $\eta(x)=0$ for $|x|\geq 2R,$ $0<\eta(x)<1$ for $R<|x|< 2R,$ and $|\partial^\alpha\eta(x)|<C(R,\alpha)$ for any $0\leq|\alpha|\leq k+1$. Notice that $B_{2R}\cap\Omega$ is a bounded domain and $\eta v=0$ on $\partial B_{2R}\cup \partial \Omega$, and $\eta v$ satisfies
 \be\ba\label{stokes2}
\begin{cases}
-\Delta (\eta v)  +\nabla (\eta\pi)=\eta f+\div (\eta F)+\tilde{f}, \,\, &x\in\Omega\cap B_{2R}, \\
\div (\eta v)=\nabla\eta\cdot v\,\, &x\in\Omega\cap B_{2R},
\end{cases}
\ea\ee
where $\tilde{f}=-\nabla\eta\cdot\nabla v-\Delta \eta v+\pi\nabla\eta-\nabla\eta F$.

Consequently, by Lemma \ref{eid1},
\be\la{vpiq1}\ba  &\|\na (\eta v)\|_{L^q}+ \|\eta\pi-\overline{\eta\pi}\|_{L^q} =\|\nabla (\eta v)\|_{L^q(B_{2R}\cap\Omega)}+\|\eta\pi\|_{L^q(B_{2R}\cap\Omega)} \\&\leq C(\|\eta f\|_ {W^{-1,q}(B_{2R}\cap\Omega)}+\|\eta F\|_{{L^q}(B_{2R}\cap\Omega)}+\|\tilde{f}\|_{W^{-1,q}(B_{2R}\cap\Omega)}+\|\nabla\eta\cdot v\|_{W^{-1,q}(B_{2R}\cap\Omega)})\\
&\leq C(\| f\|_ {L^{r(q)}}+\|F\|_{L^q}+\|\na v\|_{L^2}+ \|\pi\|_{L^2})+\frac{1}{4}(\|\na v\|_{L^q}+ \|\pi\|_{L^q}).\ea\ee
Similarly, we can check that $(1-\eta) v$ satisfies Stokes problem in $\r^3$ which one replaces $\eta$ by $1-\eta$ in \eqref{stokes2}. As a result, by Lemma \ref{wse1}, a same analysis gives
\be\la{vpiq2}\ba  &\|\na ((1-\eta) v)\|_{L^q}+ \|(1-\eta)\pi\|_{L^q} \\
&\leq C(\| f\|_ {L^{r(q)}}+\|F\|_{L^q}+\|\na v\|_{L^2}+ \|\pi\|_{L^2})+\frac{1}{4}(\|\na  v\|_{L^q}+ \|\pi\|_{L^q}).\ea\ee
Together with \eqref{vpiq1} and \eqref{vpiq2}, we get \eqref{sse11} and finish the proof.
\end{proof}
\begin{lemma} \cite[Theorem 3.5]{AACG2019}\la{ein1} Assume that $\Omega$ is a  smooth bounded  domain but not axially symmetric in $\r^3$, $1<q<+\infty$, for the problem \eqref{stokes1} with \eqref{Navi}, it holds that
\be \la{ssin1}\|\na u\|_{L^q}+ \|\pi-\bar{\pi}\|_{L^q}\le  C(q,\Omega)  \left(\|f\|_ {L^{r(q)}}+ \|F\|_{L^q}+\|\chi\|_{L^q}\right),\ee
where $r(q)>1$ is defined by \eqref{rq1}.
\end{lemma}
\begin{lemma} \la{een1} Let $\Omega$ be the exterior of a simply connected smooth bounded domain in $\r^3$, $1<q<+\infty$ and  $r(q)$ is defined by \eqref{rq1}. For the problem \eqref{stokes1} with \eqref{Navi}, \eqref{inf1} and $\chi=0$,  we have the following conclusions:

(1) If $f\in L^\frac{6}{5}$ and $F\in L^2$, there exists a unique solution $v\in W^{1,2}$ such that
\be \la{ssen1}\|\na v\|_{L^2}+ \|\pi\|_{L^2}\le  C(\Omega)  \left(\|f\|_ {L^\frac{6}{5}}+ \|F\|_{L^2}\right).\ee

(2) For any $\frac{3}{2}<q<+\infty$, if $F\in L^q$, and $f\in L^{r(q)}$ with $r(q)=\frac{3q}{q+3}$, then
\be \la{ssen11}\|\na v\|_{L^q}+ \|\pi\|_{L^q}\le  C(q,\Omega)  \left(\|f\|_ {L^{r(q)}}+ \|F\|_{L^q}+\|\na v\|_{L^2}+ \|\pi\|_{L^2}\right).\ee
.

(3) If $f\in L^p\cap L^\frac{6}{5}$, $p\in [2,6]$, and $F=0$, then
\be \la{ssen2}\|\na^2 v\|_{L^p}+ \|\na \pi\|_{L^p}\le  C(p,\Omega)  (\|f\|_ {L^p}+\|f\|_ {L^\frac{6}{5}}).\ee
\end{lemma}
\begin{proof}
For any $\phi\in D^1$ with $\phi\cdot n=0$ on $\partial\Omega$, multiplying $\eqref{stokes1}_1$ by $\phi$, we get
\be \la{ene1}2\int D(v)\cdot D(\phi)dx+\int\na \pi\cdot\phi dx=\int f\cdot \phi dx-\int F\cdot \na\phi dx.\ee
Set $\phi=v$ in \eqref{ene1}, notice that $\|\nabla v\|_{L^2}=2\|D (v)\|_{L^2}$, it is easy to check that
\be \la{ssen111}\|\na v\|_{L^2}\le  C(\Omega)  \left(\|f\|_ {L^\frac{6}{5}}+ \|F\|_{L^2}\right).\ee
On the other hand, by \eqref{ene1} and Ne\v{c}as's imbedding theorem, we obtain for any $q\in(1,\infty)$,
\be \la{ssen112}\|\pi\|_{L^q}\le  C(q,\Omega)  \left(\|f\|_ {L^\frac{6}{5}}+ \|F\|_{L^2}+\|\na v\|_{L^q}\right),\ee
which, together with \eqref{ssen111}, implies that
\be \la{ssen113}\|\pi\|_{L^2}\le  C  \left(\|f\|_ {L^\frac{6}{5}}+ \|F\|_{L^2}\right).\ee
As a result, \eqref{ssen1} is established.

Similar to the proof of \eqref{sse11} in Lemma \ref{eed1}, and along with Lemma \ref{ein1} instead of \ref{eid1}, one can get \eqref{ssen11}.

Now we will claim \eqref{ssen2}. First, set $A\triangleq-2D(n)$, \eqref{Navi} is equivalent to
\be \la{ch101} u \cdot n = 0, \,\,\curl u\times n=-A\, u \,\,\,\text{on}\,\,\, \partial\Omega.\ee
Define $(Av)^\perp\triangleq -Av\times n$, it is obvious that $(\curl u+(Au)^\perp)\times n=0$ on $\partial\Omega$. One can find that
$$\int\nabla \pi\cdot\nabla\eta dx=\int\left(f-\nabla\times(Av)^\perp\right)\cdot\nabla\eta dx,\,\,\forall\eta\in C_0^{\infty}(\r^3).$$
Thanks to \cite[Lemma 5.6]{ANIS}, for any $q\in (1,\infty)$,
\be\la{ssen114}\ba
\|\nabla \pi\|_{L^q}&\leq C(\|f\|_{L^q}+\|\nabla\times(Av)^\perp\|_{L^q})\\
&\leq C(\|f\|_{L^q}+\|\nabla v\|_{L^q}).
\ea\ee
One rewrites $\eqref{stokes1}_1$ with $F=0$ as $\nabla\times\curl v=f-\nabla \pi.$
Since $(\curl v+(Av)^\perp)\times n=0$ on $\partial\Omega$ and $\div(\nabla\times\curl v)=0$, by Lemma \ref{crle2}, we give
\be\la{x267}\ba
\|\nabla\curl v\|_{L^q}&\leq C(\|\nabla\times\curl v\|_{L^q}+\|\nabla v\|_{L^q})\\
&\leq C(\|f\|_{L^q}+\|\nabla v\|_{L^q}),
\ea\ee
On the other hand,  by Lemma \ref{crle1}, \eqref{ssen111} and \eqref{x267}, for any $p\in [2,6]$,
\bnn\ba
 \|\nabla v\|_{L^p}  &\le C (\|\curl v\|_{L^p}+\|\nabla v\|_{L^2}) \\
&\le C (\|\na\curl v\|_{L^2}+\|\nabla v\|_{L^2})\\
&\le  C  (\|f\|_ {L^2}+\|f\|_ {L^\frac{6}{5}}).
\ea\enn
Therefore, by \eqref{x267} again,
\be\la{x268}\ba
\|\nabla\curl v\|_{L^p}&\leq C(\|\nabla\times\curl v\|_{L^p}+\|\nabla v\|_{L^p})\\
&\leq C(\|f\|_{L^p}+\|f\|_ {L^\frac{6}{5}}),
\ea\ee
which, by Lemma \ref{crle3}, indicates that
\be\la{x269}\ba
\|\nabla v\|_{W^{1,p}}&\leq C(\|\curl v\|_{W^{1,p}}+\|\nabla v\|_{L^2})\\
&\leq C(\|f\|_{L^p}+\|f\|_ {L^\frac{6}{5}}).
\ea\ee
Together \eqref{ssen114} and \eqref{x269}, we arrive at \eqref{ssen2} and finish the proof.
\end{proof}
The following regularity results on the Stokes equations will be useful for our derivation of higher order a priori estimates.
\begin{lemma} \la{stokes} For   positive constants $\xmu,\smu,$ and $ q\in (3 ,\infty)$, in addition to \eqref{n3},  assume that    $\mn$ satisfies \be\la{ij80}\na\mn\in L^q,\quad
  0<\xmu\le \mn\le\smu<\infty. \ee
   For the problem with the boundary condition \eqref{Dir1} or \eqref{Navi}
\be\label{3rd1}
\begin{cases}
 -\div(2\mn D(u))  +\nabla P=F,\,\,\,\,&x\in \Omega,\\
 \div u=0,   \,\,\,&x\in  \Omega,\\
u(x)\rightarrow0,\,\,\,\,&|x|\rightarrow\infty,
\end{cases}
\ee
we have the following conclusions:

(a) If $F=f\in L^{6/5}\cap L^r$ with   $r\in[ 2q/(q+2),q],$ then there exists some positive constant $C$ depending only on $ \xmu , \smu , r, $ and $q$  such that the unique weak solution $(u,P)\in D^1_{0,\div}\times L^2$  satisfies
 \be\ba\label{3rd2}
\|\na u\|_{L^2 }+\|P\|_{L^2 }\le C \|f\|_{L^{6/5} },
\ea\ee
 \be\ba\label{3rd3'}
\|\na^2 u\|_{L^r}+\xl\|\na P\xr\|_{L^r}\le  C  \|f\|_{L^r}+ C \left(\|\na\mn\|_{L^q}^{\frac{q(5r-6)}{2r(q-3)}}+1\right)\|f\|_{L^{6/5}} .
\ea\ee

(b) If $F=\div g$ with $g\in L^2\cap L^{\ti r}$ for some $\ti r\in (6q/(q+6),q],$ then there exists a positive constant $C$
 depending only on $\xmu, \smu, q,$ and $\ti r$ such that the unique weak solution $(u,P)\in D^1_{0,\div}\times L^2$ to
 \eqref{3rd1}  satisfies\be \la{3e1}\|\na u\|_{L^2\cap L^{\ti r}}+ \|P\|_{L^2\cap L^{\ti r}}\le  C  \|g\|_{L^2\cap L^{\ti r} }+C \|\na\mn\|_{L^q}^{\frac{3q(\ti r-2)}{2\ti r(q-3)}}\|g\|_{L^2}.\ee\end{lemma}

\begin{proof} First, for any $\phi\in D^1$ with $\phi\cdot n=0$ on $\partial\Omega$, multiplying $\eqref{3rd1}_1$ by $\phi$, we get
\be \la{3rd31}2\int \mn D(u)\cdot D(\phi)dx+\int\na P\cdot\phi dx=\int F\cdot \phi dx.\ee
Taking $\phi=u$ and  $F=f+\div g$ in \eqref{3rd31}, and integrating by parts, we obtain after using \eqref{3rd1}$_2$ that
  \be\ba\label{3rd3}\notag
 2\int \mn|D(u)|^2dx=\int f\cdot udx+\int g\cdot \na udx\le C(\|f\|_{L^{6/5}}+\|g\|_{L^2} )\|\na  u\|_{L^2},
\ea\ee
which together with \eqref{ij80} yields
  \be\ba\label{3rd4}
\|\na  u\|_{L^2} \le C\xmu^{-1} (\|f\|_{L^{6/5}}+\|g\|_{L^2} ),
\ea\ee due to
  \be\ba\label{3rd12}
 2\int|D(u)|^2dx= \int|\na u|^2dx.
\ea\ee

Furthermore,  by \eqref{3rd31} and Ne\v{c}as's imbedding theorem, we check that for any $q\in(1,\infty)$,
\be \la{3rd32}\|P\|_{L^q}\le  C \left(\|f\|_ {W^{-1,q}}+ \|g\|_{L^p}+\|\na v\|_{L^q}\right),\ee
which together with the Sobolev inequality and \eqref{3rd12}   gives
 \be\ba\label{3rd6}\notag
 \|P\|_{L^2}  \le C(\|f\|_{L^{6/5}}+\|g\|_{L^2}).
\ea\ee
  Combining this with \eqref{3rd4} leads to \eqref{3rd2} and the part result of \eqref{3e1} on $L^2$ estimate.

Next, we will claim \eqref{3rd3'}. One can rewrite  \eqref{3rd1}$_1$ with $F=f$ as
 \be\ba\label{3rd7}
  -\Delta u+\na \xl(\frac{P}{\mn}\xr)=\frac{f}{\mn}+\frac{2D(u)\cdot\na\mn }{\mn}-\frac{P\na\mn}{\mn^2}.
\ea\ee
Applying  Lemma \ref{eed1} (or Lemma \ref{een1} ) to the Stokes system \eqref{3rd7}  yields that for  $r\in[ 2q/(q+2),q],$ $\theta=\frac{2r(q-3)}{q(5r-6)}$,
\bnn\ba\label{3rd8}
  \|\na^2 u\|_{L^r}+\xl\|\na P\xr\|_{L^r}
 &\le \|\na^2 u\|_{L^r}+C\xl\|\na  \xl(\frac{P}{\mn}\xr)\xr\|_{L^r}+C\xl\| \frac{P\na\mn}{\mn^2}\xr\|_{L^r}\\
 &\le C (\xl\|  {f}  \xr\|_{L^r}+\|f\|_{L^{6/5}})+C\xl\|  {2d\cdot\na\mn} \xr\|_{L^r}+C\xl\|  {P\na\mn} \xr\|_{L^r}  \\
   &\le C (\xl\|  {f}  \xr\|_{L^r}+\|f\|_{L^{6/5}})+C \|\na\mn\|_{L^q} \|\na u\|_{L^2}^{\theta}\|\na^2 u\|_{L^r}^{1-\theta}\\
   &\quad+C \|\na\mn\|_{L^q}\xl\| P\xr\|_{L^2}^{\frac{2r(q-3)}{q(5r-6)}}\xl\|\na P \xr\|_{L^r}^{1-\frac{2r(q-3)}{q(5r-6)}}\\
    &\le C (\xl\|  {f}  \xr\|_{L^r}+\|f\|_{L^{6/5}})+C \|\na\mn\|_{L^q}^{\frac{q(5r-6)}{2r(q-3)}} (\|\na u\|_{L^2}+  \|  {P} \|_{L^2} ) \\&\quad+\frac{1}{2}\xl(\|\na^2 u\|_{L^r}+ \xl\|\na P\xr\|_{L^r}\xr),
\ea\enn
  which combined with  \eqref{3rd2} yields    \eqref{3rd3'}.

Finally, we will prove \eqref{3e1}. Rewrite  \eqref{3rd1}$_1$  with $F=\div g$ as
 \be\ba\label{5rd7}
  -\Delta u+\na \xl(\frac{P}{\mn}\xr)=
  \div\left(\frac{g}{\mn}\right) +\ti G,
\ea\ee where $$\ti G\triangleq\frac{g\cdot\na\mn }{\mn^2}+\frac{2d\cdot\na\mn }{\mn}-\frac{P\na\mn}{\mn^2}. $$
 By Hodler's and Young's inequalities, for any $\delta>0,$ \be\la{5r3}\ba \|\ti G\|_{L^{\frac{3\ti r}{ 3+\ti r}}}&\le \delta (\|g\|_{L^{\ti r}}+\|\na u\|_{L^{\ti r}}+\|P\|_{L^{\ti r}})\\&\quad+C(\delta)\|\na\mn\|_{L^q}^{\frac{3q(\ti r-2)}{2\ti r(q-3)}}( \|g\|_{L^2}+\|\na u\|_{L^2}+\|P\|_{L^2}).\ea\ee
By  Lemma \ref{eed1} (or Lemma \ref{een1}), choosing $\delta$ suitably small,  we get
\bnn\ba \|\na u\|_{L^{\ti r}}&\le C\left(\|\frac{g}{\mn}\|_{L^{\ti r}}+\|\ti G\|_{L^{\frac{3\ti r}{ 3+\ti r}}}+\|\na u\|_{L^2}+\|P\|_{L^2}\right)\\ &\le C\|g\|_{L^{\ti r}}+C\|\ti G\|_{L^{\frac{3\ti r}{3+\ti r}}},\ea\enn which together with \eqref{3rd32} yields  \bnn\ba \|\na u\|_{L^{\ti r}}+\|P\|_{L^{\ti r}} \le C(\|g\|_{L^{\ti r}}+\|g\|_{L^2})+C\|\ti G\|_{L^{\frac{3\ti r}{3+\ti r}}}.\ea\enn  Combining this and \eqref{5r3} gives \eqref{3e1}. The proof of Lemma \ref{stokes} is finished.
\end{proof}

\section{A Priori Estimates}\label{sec3}
In this section, $\Omega$ is always the exterior of a simply connected bounded smooth domain $D$ in $\r^3$, we will establish some necessary a priori bounds of local strong solutions $(\rho,u,P)$ to the system \eqref{1.1}--\eqref{inf1}  whose existence is guaranteed by Lemma \ref{local}. Thus, let $T>0$ be a fixed time and $(\rho, u,P)$  be the smooth solution to \eqref{1.1}-\eqref{inf1} on $\Omega\times(0,T]$ with smooth initial data $(\rho_0,u_0)$ satisfying \eqref{2.2}.

We have the following key a priori estimates on $(\n,u,P)$.
\begin{proposition}\la{pr1}  There exists some  positive constant  $\ve_0$
    depending    only on  $q,  \bar\n, \xmu, \smu, $ $\|\n_0\|_{L^{3/2}}, $ and $M$  such that if
       $(\n,u,P)$  is a smooth solution of
       \eqref{1.1}--\eqref{inf1}  on $\Omega \times (0,T] $
        satisfying
 \be\la{3a1}
     \sup_{t\in[0,T]}\|\na\mn\|_{L^q}\le 4M ,\quad  \int_0^T\|\na u\|_{L^2}^4dt \leq 2\|\na u_0\|_{L^2}^2 ,
   \ee
 the following estimates hold
        \be\la{3a2}
\sup_{t\in[0,T]}\|\na\mn\|_{L^q} \le 2M ,\quad \int_0^T\|\na u\|_{L^2}^4dt \leq  \|\na u_0\|_{L^2}^2 ,
  \ee
   provided $\|\na u_0\|_{L^2}\le \ve_0.$
\end{proposition}

Before proving Proposition \ref{pr1}, we establish some necessary a priori estimates, see Lemmas \ref{lem-ed1}--\ref{lem5.1}.

Setting
\be \la{laj1}{\sigma(t)} \triangleq \min\{1,t\},\,\,\lambda\triangleq    3\xmu/(4\|\n_0\|_{L^{3/2}}) , \ee
we start with  the following  exponential  decay-in-time  estimates on the solutions for large time, which  plays a crucial role  in our analysis.
\begin{lemma}\label{lem-ed1}Let $(\n, u, P)$ be a smooth solution to  \eqref{1.1}--\eqref{inf1} satisfying \eqref{3a1}.
Then  there exists  a generic positive constant  $C$  depending only on  $q,$ $\ban,$  $\xmu,$ $\smu,$ $\|\n_0\|_{L^{3/2}},$  and $M$ such that
\be\label{ed1} \sup_{t\in[0,T]}e^{\lambda t}\|\n^{1/2}u\|_{L^2}^2 +\int_0^Te^{\lambda t}\int |\na u|^2dxdt \le  C \|\na u_0\|_{L^2}^2,\ee
\be\label{ed2}   \sup_{t\in[0,T]}e^{\lambda t}\int |\na u|^2dx +\int_0^T e^{\lambda t}\int\rho|u_t|^{2}dxdt \le C  \|\na u_0\|_{L^2}^2, \ee
\be\label{ed3}   \sup_{t\in[0,T]}\sigma e^{\lambda t}\int\rho|u_t|^{2}dx +\int_0^T \sigma e^{\lambda t}\int  |\na u_t|^2dx dt \le C  \|\na u_0\|_{L^2}^2, \ee
\be\label{3d6.7}     \ba& \sup_{t\in[0,T]}\sigma e^{\lambda t}\xl(\|\na u\|_{H^1}^2+ \|P\|_{H^1}^2 \xr)+\int_0^Te^{\lambda t}\xl( \|\na u\|_{H^1}^2+ \|P\|_{H^1}^2 \xr)dt \leq C\|\na u_0\|_{L^2}^2,\ea\ee
and for any $p\in [2, \min\{6,q\}]$,
\be\label{3d6.1}  \ba \int_0^T \sigma e^{\lambda t}\left(\|\na u\|_{  W^{1, p}}^2
+ \|P\|_{  W^{1, p}}^2\right)dt\leq C(p)\|\na u_0\|_{L^2}^2.\ea\ee
\end{lemma}

\begin{proof}
First, standard arguments (\cite{L1996}) indicate that \begin{equation}\label{gj0}
0\le\n\leq \bar\n,~~~~~\|\n\|_{L^{3/2}}= \|\n_0\|_{L^{3/2}},
\end{equation}
and then by \eqref{3rd12},
\be\ba\label{ed1.2}
 \|\n^{1/2}u\|_{L^2}^2\le \|\n\|_{L^{3/2}}\|u\|_{L^6}^2\le  \frac{4}{3}\|\n_0\|_{L^{3/2}}\|\na u\|_{L^2}^2\le    \lambda^{-1}\int 2\mn |D(u)|^2dx,
\ea\ee
with  $\lambda$ is defined as in \eqref{laj1}.

Multiplying \eqref{1.1}$_2$ by $u$ and integrating by parts, we have
\be\ba\label{ed1.1}
 \frac{1}{2}\frac{d}{dt}\|\n^{1/2}u\|_{L^2}^2+\int 2\mn |D(u)|^2dx=0,
\ea\ee
 with together with \eqref{ed1.2} yields
$$  \frac{d}{dt}\|\n^{1/2}u\|_{L^2}^2+\lambda \|\n^{1/2}u\|_{L^2}^2+\int 2\mn |D(u)|^2dx\le 0.
$$
Therefore,
\bnn\ba&
 \sup_{t\in[0,T]}e^{\lambda t}\|\n^{1/2}u\|_{L^2}^2 +\int_0^Te^{\lambda t}\int |\na u|^2dxdt\\ &\le C \|\n_0^{1/2}u_0\|_{L^2}^2 \le C\|\n_0\|_{L^{\frac{3}{2}}}\|u_0\|_{L^6}^2  \le  C \|\na u_0\|_{L^2}^2,
\ea\enn
and \eqref{ed1} is established.

Next, it follows from  Lemma \ref{stokes}, \eqref{1.1}$_1$, \eqref{3a1},   \eqref{gj0}, and Gagliardo-Nirenberg's inequality that
\bnn
\ba
 \|\na  u\|_{H^1}+ \|  P \|_{H^1}
  &\le C \xl(\|\n  u_t+\n u\cdot\na u\|_{L^2}+\|\n  u_t+\n u\cdot\na u \|_{L^{6/5}}\xr)\\
   &\le C ( \ban^{1/2}+\|\n\|^{1/2}_{L^{3/2}})\xl(\|\n^{1/2} u_t\|_{L^2}+\ban^{1/2} \|  u\cdot \na  u\|_{L^2}\xr) \\
 &\le C  \|\n^{1/2} u_t\|_{L^2}+C \|\na u\|_{L^2} \|\na u\|_{L^2}^{1/2}\|\na^2 u\|_{L^2}^{1/2}\\
      &\le C \|\n^{1/2} u_t\|_{L^2}+C \|\na u\|_{L^2}^3 +\frac12 \|\na^2 u\|_{L^2},
\ea\enn which immediately leads to
\begin{equation} \label{3d4}
\ba
 &\|\na u\|_{H^1}+ \|  P \|_{H^1}+\|\n   u_t+\n u\cdot\na u \|_{L^{6/5}\cap L^2}\\
   &\le C \|\n^{1/2} u_t\|_{L^2}+C \|\na u\|_{L^2}^3.
\ea\end{equation}
Multiplying \eqref{1.1}$_1$ by $u_t$, and by \eqref{3d4}, we get
\begin{equation}\label{z1.3}
\ba
& \frac{d}{dt}\int \mn |D(u)|^2dx+\int\rho|u_t|^{2}dx\\
&=-\int\n u\cdot\na u\cdot u_tdx+ \int  \mn u\cdot\na |D(u)|^2dx\\
&\le \ban^{1/2}\|\n^{1/2}u_t\|_{L^2}\|u\|_{L^6}\|\na u\|_{L^3}+C\smu \|u\|_{L^6}\|\na u\|_{L^3}\|\na^2 u\|_{L^2} \\
&\le C \|\n^{1/2}u_t\|_{L^2}  \|\na u\|_{L^2}  \|\na u\|_{L^2}^{1/2} \|\na^2 u\|_{L^2}^{1/2}   +C  \|\na u\|_{L^2} \|\na u\|_{L^2}^{1/2}\|\na^2 u\|_{L^2}^{3/2} \\
&\le \frac{1}{2}\|\n^{1/2}u_t\|_{L^2}^2+C \|\na u\|_{L^2}^6,
\ea\end{equation}
which implies that
  \be\label{3d5}
\ba
& \frac{d}{dt}\int  2\mn  |D(u)|^2dx+\int\rho|u_t|^{2}dx\le C  \|\na u\|_{L^2}^4\|\na u\|_{L^2}^2.
\ea\ee
Hence, by Gr\"onwall's inequality, \eqref{3a1}  and \eqref{ed1}, we arrive at \eqref{ed2}.

Now we will claim \eqref{ed3}. Imposing $u_t\cdot\partial_{t}$ on \eqref{1.1}$_2$, using integration by parts and \eqref{1.1}$_1$, we check that
\be\ba\label{3d2.1}&
\frac12\frac{d}{dt}\int\rho|u_t|^{2}dx+\int2\mn|D(u_t)|^{2}dx   \\
&=-\int (2\n u\cdot\na u_t\cdot u_t+\n u_t\cdot\na u\cdot u_t)dx-\int\n u\cdot\na(u\cdot\na u\cdot u_t)dx\\
&\quad+2\int \xl(u\cdot\na\mn\xr) D(u)\cdot D(u_t)dx \triangleq \sum_{i=1}^3I_i.
\ea\ee

 On the other hand, by Gagliardo-Nirenberg's and Young's inequalities, \eqref{3a1} and \eqref{gj0},
\be\ba\label{3d2.2}
 | I_1|
  &\le C \|\n^{1/2}u_t\|_{L^3}\|\na u_t\|_{L^2}\|u\|_{L^6}+C \|\n^{1/2}u_t\|_{L^3}\|\na u\|_{L^2}\|u_t\|_{L^6} \\
  &\le C \|\n^{1/2}u_t\|_{L^2}^{1/2}\|\na u_t\|_{L^2}^{3/2}\|\na u \|_{L^2}\\
    &\le  \frac{1}{4}\xmu\|\na u_t\|_{L^2}^2+C \|\n^{1/2}u_t\|_{L^2}^2\|\na u \|_{L^2}^4,
\ea\ee
\be\ba\label{3d2.3}
  |I_2| &=\left|\int\n u\cdot\na(u\cdot\na u\cdot u_t)dx\right|\\
  &\le C\int \n|u||u_t|\left(|\na u|^2+|u| |\na^2u|\right)dx+\int\n|u|^2|\na u||\na u_t|dx\\
  &\le C\|u\|_{L^6}\|u_t\|_{L^6} \xl( \|\na u\|_{L^3}^2+   \|u\|_{L^6} \|\na^2 u\|_{L^2} \xr)+ C  \|u\|_{L^6}^2\|\na u\|_{L^6}\|\na u_t\|_{L^2}\\
  &\le C  \|\na u_t\|_{L^2}\|\na^2 u\|_{L^2}\|\na u\|_{L^2}^{2}\\
    &\le  \frac{1}{8}\xmu\|\na u_t\|_{L^2}^2+C \|\na^2 u\|_{L^2}^2 \|\na u\|_{L^2}^4,
\ea\ee and
\be\ba\label{3d2.5}
  |I_3 |   &\le  C\|\na\mn\|_{L^q} \|u\|_{L^\infty} \|\na u_t\|_{L^2}\|\na u\|_{L^{\frac{2q}{q-2}}}\\    &\le  C(q,M)\|u\|_{L^6}^{1/2}\|\na u\|_{L^6}^{1/2} \|\na u_t\|_{L^2}\|\na u\|_{L^2}^{\frac{q-3}{q}}\|\na^2 u\|_{L^2}^\frac{3}{q} \\    &\le  \frac{1}{8}\xmu\|\na u_t\|_{L^2}^2+C \|\na u \|_{L^2} \|\na^2 u\|_{L^2}^3+ C \|\na u \|_{L^2}^4.
\ea\ee

  Substituting \eqref{3d2.2}--\eqref{3d2.5} into \eqref{3d2.1}, and by \eqref{3d4}, we conclude that
\be\ba\label{3d2.6}&
 \frac{d}{dt}\int\rho|u_t|^{2}dx+\xmu\int|\nabla u_t|^{2}dx   \\
&\le  C \left(\|\rho^{1/2} u_t\|_{L^2}^2 +   \|\na^2 u\|_{L^2}^2 \right)\|\na u\|_{L^2}^4+C \|\na u \|_{L^2} \|\na^2 u\|_{L^2}^3+ C \|\na u \|_{L^2}^4 \\
&\le C \|\rho^{1/2}u_t\|_{L^2}^2 \|\na u \|_{L^2}^4 +C \|\rho^{1/2}u_t\|_{L^2}^3\|\na u \|_{L^2}  +C \|\na u \|_{L^2}^{10} +C \|\na u \|_{L^2}^{2},
\ea\ee
which, along with \eqref{ed2} gives
 \be\ba\label{3d2.6'}&
 \frac{d}{dt}\int\rho|u_t|^{2}dx+\xmu\int|\nabla u_t|^{2}dx
\\
&\le C \|\rho^{1/2}u_t\|_{L^2}^2 \xl(\|\na u \|_{L^2}^4 + \|\rho^{1/2}u_t\|_{L^2} \|\na u \|_{L^2}\xr)\\&\quad +C  \|\na u \|_{L^2}^{4}+C \|\na u \|_{L^2}^{2}.
\ea\ee
As a result, by Gr\"onwall's inequality, \eqref{ed1}, and \eqref{ed2}, we prove \eqref{ed3}.

Obviously , \eqref{3d6.7} comes from a combination of \eqref{3d4} and \eqref{ed1}-- \eqref{ed3}.

Finally, by Gagliardo-Nirenberg's inequality, we find that for any $   p\in [2,\min\{6,q\}],$
\bnn\label{nz1.4}
\ba&
\|\n u_t+\n u\cdot\na u\|_{L^p} \\
   &\le C\|\n^{1/2} u_t\|_{L^2}^{\frac{6-p}{2p}}\|\n^{1/2} u_t\|_{L^6}^{\frac{3p-6}{2p}}+C\|  u\|_{L^6}\|\na u\|_{L^{\frac{6p}{6-p}}}\\
      &\le C\|\n^{1/2} u_t\|_{L^2}^{\frac{6-p}{2p}}\|\na u_t\|_{L^2}^{\frac{3p-6}{2p}}
      + C\|\na u\|_{L^2}\|\na u\|_{L^2}^{\frac{p}{5p-6}}\|\na^2 u\|_{L^p}^{\frac{4p-6}{5p-6}},
\ea\enn
which together with     \eqref{3rd3'}  and  \eqref{3d4}  shows
\be\ba\label{3d4.21}       \|\na^2u\|_{L^p} +\|\na P\|_{L^p}
 &\le C \|\n u_t+\n u\cdot\na u\|_{L^{6/5}\cap L^p}   \\
&\le   C\|\n^{1/2} u_t\|_{L^2}^{\frac{6-p}{2p}}\|\na u_t\|_{L^2}^{\frac{3p-6}{2p}}+ C\|\na u\|_{L^2}^{\frac{6p-6}{p}}\\
&\quad+\frac12\|\na^2 u\|_{L^p}+C\|\n^{1/2} u_t\|_{L^2}+C \|\na u\|_{L^2}^3 .\ea\ee
Consequently,
\be\ba\label{3d4.2}       \|\na^2u\|_{L^p} +\|\na P\|_{L^p}
&\le   C\|\n^{1/2} u_t\|_{L^2}^{\frac{6-p}{2p}}\|\na u_t\|_{L^2}^{\frac{3p-6}{2p}}+ C\|\na u\|_{L^2}^{\frac{6p-6}{p}}\\
&\quad + C\|\n^{1/2} u_t\|_{L^2}+C \|\na u\|_{L^2}^3.\ea\ee
Along with  \eqref{ed1}-- \eqref{ed3}, we derive \eqref{3d6.1} and complete the proof of Lemma \ref{lem-ed1}.
\end{proof}

In order to obtain  the $L^\infty(0,T;L^q)$-norm of the  gradient of $\mn$, we need  to prove the following conclusion which could be regarded as a direct result of  Lemma \ref{lem-ed1}.

\begin{lemma}\label{lem5.1} Let $(\n, u, P)$ be a  smooth  solution to  \eqref{1.1}--\eqref{inf1} satisfying \eqref{3a1}.  Then there exists  a generic positive constant  $C$  depending only on  $q,$ $\ban,$ $\xmu,$ $\smu,$ $\|\n_0\|_{L^{3/2}},$  and $M$ such that
\be\label{3d4.1}    \int_0^T \|\na u\|_{L^\infty}dt \le C  \|\na u_0\|_{L^2} . \ee
\end{lemma}

\begin{proof} Setting \be\la{zv1}r\triangleq\frac12\min\{q+3,9\}\in\xl(3,\min\xl\{ q,6\xr\}\xr),\ee
By Sobolev inequality  and \eqref{3d4.2}, we check that
\be\label{3d4.6} \ba    \|\na u\|_{L^\infty}  & \le C\|\na  u\|_{L^2}+C\|\na^2 u\|_{L^r}\\ &\le C\|\n^{1/2}u_t\|_{L^2}^\frac{6-r}{2r}\|\na u_t\|_{L^2}^\frac{3r-6}{2r}+C \|\n^{1/2}u_t\|_{L^2}+C\|\na u\|_{L^2}^\frac{6(r-1)}{r}+C\|\na u\|_{L^2}\\
&\le C\sigma^{-\frac{1}{2}}e^{-\lambda t/2}(\sigma e^{\lambda t}\|\n^{1/2}u_t\|_{L^2}^2)^\frac{6-r}{4r}(\sigma e^{\lambda t}\|\nabla u_t\|_{L^2}^2)^\frac{3r-6}{4r}\\
&\quad +Ce^{-\lambda t/2} \left(\|\na u_0\|_{L^2}+(e^{\lambda t}\|\n^{1/2}u_t\|_{L^2}^2)^\frac{1}{2}\right),\ea\ee

which along with \eqref{3a1}, \eqref{ed2}, and \eqref{ed3} implies that
\bnn\ba    & \int_0^T \|\na u\|_{L^\infty}dt \\
&\le   C \|\na u_0\|_{L^2}^\frac{6-r}{2r} \xl(\int_0^T \sigma^{-\frac{2r}{r+6}}e^{-\frac{2\lambda r}{r+6}t}dt\xr)^\frac{r+6}{4r} \xl(\int_0^T\sigma e^{\lambda t}\|\nabla u_t\|_{L^2}^2dt\xr)^\frac{3r-6}{4r}+C \|\na u_0\|_{L^2}\\
&\le  C \|\na u_0\|_{L^2}.\ea\enn
\end{proof}

With  Lemmas \ref{lem-ed1}--\ref{lem5.1} at hand, it's time to give the proof of  Proposition \ref{pr1}.

\emph{\bf Proof of Proposition \ref{pr1}.}  Since $\mn$ satisfies
\be\label{3d5.4}   ( \mn)_t+  u\cdot\na  \mn=0,\ee and then
\be\label{3d5.5}   \frac{d}{dt}\|\na\mn\|_{L^{q}}\le q\|\na u\|_{L^\infty}\|\na\mn\|_{L^{q}},\ee
which, by Gr\"onwall's inequality and \eqref{3d4.1}, yields
\be \la{cbd2z}\ba    \sup_{t\in[0,T]} \|\na\mn\|_{L^{q}}&\le \|\na\mu(\n_0)\|_{L^{q}} \exp\xl\{ q\int_0^T \|\na u\|_{L^\infty}dt \xr\}\\
&\le \|\na\mu(\n_0)\|_{L^{q}} \exp\xl\{  C \|\na u_0\|_{L^2}\xr\}\\
&\le 2\|\na\mu(\n_0)\|_{L^{q}},
\ea\ee
provided
\be\ba\label{3d5.6}   \|\na u_0\|_{L^2} \le\ve_1\triangleq   C^{-1} \ln 2.
\ea\ee

 Moreover, it follows from  \eqref{ed2} that
\begin{equation}\label{3a1.2} \ba
 \int_0^T\|\na u\|_{L^2}^4dt
 \le& \sup_{t\in[0,T]}\xl(e^{\lambda t}\|\na u\|_{L^2}^2\xr)^2\int_0^Te^{-2\lambda t}dt\\
 \le & C  \|\na u_0\|_{L^2}^4\le  \|\na u_0\|_{L^2}^2,
\ea\end{equation}
provided
\be\ba\label{3cb5.6}    \|\na u_0\|_{L^2} \le\ve_2\triangleq   C^{-1/2}  .
\ea\ee

Choosing $\ve_0\triangleq \min\{1,\ve_1,\ve_2\},$ and by \eqref{cbd2z}--\eqref{3cb5.6}, we establish \eqref{3a2} .\hfill$\Box$

In the following lemma, we will continue to investigate some higher-order estimates of the system \eqref{1.1}.

\begin{lemma}\label{lem5.a3}  Under the assumptions of Proposition \ref{pr1}, there exists a positive constant $C$ depending only on $q$, $\ban$, $\xmu$, $\smu$,  $\|\n_0\|_{L^{3/2}}$, $M$, and $\|\na\n_0\|_{L^2} $  such that for $ p_0\triangleq\min\{6,q\} $ and $ q_0\triangleq 2q/(q-3)$,
\be\label{3dr6.1}  \ba& \sup_{t\in[0,T]}\sigma^{q_0}   e^{\lambda t}\xl(\|\na u\|^2_{W^{1,p_0}}+ \|P\|^2_{W^{1,p_0}}+ \|\na u_t\|^2_{L^2}\xr)   \\&
 +\int_0^T\sigma^{q_0} e^{\lambda t} \xl(\sigma^\frac{1}{2}\|(\n  u_{t})_t \|_{L^2}^2+  \|\na u_t\|_{  L^{ p_0}}^2
+  \|P_t\|_{L^2\cap L^{ p_0}}^2\xr)dt\leq C,\ea\ee
provided $\|\na u_0\|_{L^2}<\ve_0$, where $\ve_0$ is given in Proposition \ref{pr1}.
\end{lemma}

\begin{proof} Similar to the way we get \eqref{3d5.5} and \eqref{cbd2z}, we also have
\be\ba\label{3d6.21}   \sup_{0\le t\le T}\|\na\n\|_{L^2 } \le 2\|\na\n_0\|_{L^2 }.
\ea\ee   Hence, by Sobolev's inequality and \eqref{ed2},
\be \la{ij6}\ba\|\n_t\|_{L^2\cap L^{3/2}}&=\|u\cdot\na\n\|_{L^2\cap L^{3/2}}\\&\le C\|\na\n\|_{L^2}\|\na u\|_{L^2}^{1/2}\|\na u\|_{H^1}^{1/2} \le C\|\na u\|_{H^1}^{1/2} .\ea\ee

Next, it follows from \eqref{1.1}$_2$ that $u_t$ satisfies
\bnn\label{5rd3}
\begin{cases}
 -\div(2\mn D(u_t))  +\nabla P_t=\ti F+\div g,\\
 \div u_t=0,
\end{cases}
\enn  with the same boundary condition and the far field behavior as $u$, where\bnn  \ti F\triangleq  -
\n u_{tt}- \n u \cdot\na u_t-\n_tu_t-(\n u)_t\cdot\na u, \quad g\triangleq-2u\cdot\na\mn  D(u) . \enn
 Therefore, a straightforward application  of  Lemma \ref{stokes} indicates that
 \be \la{3'd.1}\ba  \|\na u_t\|_{ L^2\cap L^{p_0}}+\|P_t\|_{ L^2\cap L^{p_0}}     \le C_1\|\ti F\|_{L^{6/5}\cap L^{\frac{3p_0}{p_0+3}}}+C\|g\|_{L^2\cap L^{p_0}}.\ea\ee
 On the other hand, for the two terms on the
righthand side of \eqref{3'd.1},  we deduce from \eqref{3a1}, \eqref{gj0}, \eqref{ij6} and \eqref{ed2} that
 \be\la{ij1}\ba   &C_1\|\ti F\|_{L^{6/5}\cap L^{\frac{3p_0}{p_0+3}}}  \\ &\le  C\|\n\|^{1/2}_{L^{3/2}\cap L^{\frac{3p_0}{6-p_0}}}\|\n^{1/2} u_{tt} \|_{L^2}+C\|\n\|_{L^{3}\cap L^{\frac{6p_0}{6-p_0}}}\|u\|_{L^\infty}\|\na u_t\|_{L^2}\\&\quad+ C\| \n_t\|_{L^2\cap L^{3/2}} \xl(\|u_t\|_{L^6\cap L^{\frac{6p_0}{6-p_0}} } +\|\na u\|_{H^1}^\frac{3}{2}+\|\na u\|_{H^1}^\frac{1}{2} \|\na u\|_{W^{1,p_0}}\xr) \\&\quad+C\|\n\|_{L^2\cap L^{p_0}}\| u_t\|_{L^6}\|\na u \|_{L^6}\\ &\le
 C \| \sqrt{\n}  u_{tt} \|_{L^2}+\frac{1}{2} \|\na u_t\|_{L^{p_0}}+C \|\na u_t\|_{L^2}(1+\|\na u\|_{H^1} ) +C\|\na u\|_{H^1}^2,\ea\ee and
\be\la{ij2}\ba \|g\|_{L^2\cap L^{p_0}}&\le C\|\na \mn\|_{L^q}\| u\|_{L^6\cap L^\infty}\|\na u\|_{L^2\cap L^\infty}\\& \le  C\|\na u_t\|_{L^2} \|\na u\|_{H^1}+C \|\na u\|_{H^1}^2,\ea\ee
where in the second inequality in \eqref{ij1} and \eqref{ij2} we have utilized the following simple fact
\be\label{3d6.5}     \|\na u\|_{H^1\cap W^{1,p_0}} +\|P\|_{H^1\cap W^{1,p_0}}\le  C\|\na u_t\|_{L^2}  +C \|\na u \|_{L^2}\ee
due to \eqref{3d4}, \eqref{3d4.2}, \eqref{ed2} and \eqref{gj0}.

As a consequent, adding  \eqref{ij1} and \eqref{ij2} into \eqref{3'd.1}, we conclude that
 \be \la{3d.1}\ba &\|\na u_t\|_{ L^2\cap L^{p_0}}+\|P_t\|_{ L^2\cap L^{p_0}}  \\ &\le
 C \| \sqrt{\n}  u_{tt} \|_{L^2} +C \|\na u_t\|_{L^2}(1+\|\na u\|_{H^1} ) +C\|\na u\|_{H^1}^2.\ea\ee

Now, operating $u_{tt}\cdot\partial_{t}$ to \eqref{1.1}$_2$,  together with \eqref{1.1}$_1$ and \eqref{3d5.4}, we have
\be\la{utt1}\ba  &\int \n |u_{tt} |^2dx
=-\int\n_t u_t u_{tt} dx-\int\n_t u\cdot\na u\cdot u_{tt} dx-\int(2\mu(\n)D(u))_t\cdot D(u_{tt}) dx\\
&\quad +\int \n( u \cdot \na u_t+ u_t\cdot\na u)\cdot u_{tt} dx \\
&=-\frac{d}{dt}\int \mu(\n)|D(u_t)|^2dx+\theta'(t)+\frac{1}{2}\int (\n u)_t\cdot \na (|u_t|^2) dx+\int \n _t(u\cdot \na u)_t \cdot  u_t dx\\
&\quad +\int (\n u)_t\cdot \na (u \cdot \na u\cdot u_t) dx-3\int u\cdot\na\mu(\n)|D(u_t)|^2dx-\int u_t\cdot\na\mu(\n) (|D(u)|^2)_tdx\\
&\quad -2\int u\cdot\na(\mu(\n))_t D(u)\cdot D(u_t)dx+\int \n( u \cdot \na u_t+ u_t\cdot\na u)\cdot u_{tt} dx \\
&\triangleq -\frac{d}{dt}\int \mu(\n)|D(u_t)|^2dx+ \theta'(t)
+ \sum_{i=1}^{7}J_i.\ea\ee
where
$$\theta(t)=\frac{1}{2}\int \n u \cdot\na(|u_t|^2)dx-\int \n_t  u\cdot\na u \cdot u_t dx+2\int u\cdot\na\mu(\n)\,D(u)\cdot D(u_t)dx.$$
Now we have to estimate $J_i$, $1\leq i\leq 7$. First, by \eqref{ij6} and \eqref{3d.1},
\be\la{ij11}\ba|J_1|+|J_2|&\le  \|\n_t\|_{L^2}\|u \|_{L^\infty}  \|\na u_t\|_{L^3}\|u_t\|_{L^6}+\|\n\|_{L^6}\|\na u\|_{L^2}\|u_t\|_{L^6}^2\\
&\leq \delta \|\sqrt{\n} u_{tt}\|_{L^2}^2+C(\delta)(1+ \|\na u \|^2_{H^1} )\|\na u_t\|_{L^2}^2+C\|\na u \|^4_{H^1}.\ea\ee
Next, due to \eqref{ij6} \eqref{3d6.5} and \eqref{3d.1},
\be\la{ij12}\ba|J_3|&\le  C\|\n_t\|_{L^2}\|\na u\|^{1/2}_{H^1} (\|\na u\|_{H^1} \|\na u\|_{H^1\cap W^{1,p_0}}+\|\na u\|_{H^1}^2)\|\na u_t\|_{L^2}\\
&\quad +C\|\na u\|_{H^1}^2\|\na u_t\|_{L^2}^2\\
&\leq \delta \|\sqrt{\n} u_{tt}\|_{L^2}^2+C(\delta)(1+ \|\na u \|^2_{H^1} )\|\na u_t\|_{L^2}^2+C\|\na u \|^4_{H^1}.\ea\ee
For $J_4$, set $\beta\triangleq\frac{q(p_0-2)}{ p_0q-p_0-2q },$  it clear that $\beta\in [1, 2q_0]$. By \eqref{3a1}, \eqref{ed2}, and  \eqref{3d.1},
\be\la{ij4}\ba|J_4|&\le C\|u\|_{L^\infty}\|\na \mn\|_{L^q}\|\na u_t\|^{\frac{2}{\beta}}_{L^2}\|\na u_t\|^{2(1-\frac{1}{\beta})}_{L^{p_0}}\\&\le \delta\|\na u_t\|^2_{L^{p_0}}+C(\delta)\|\na u\|_{H^1}^{\frac{\beta}{2}} \|\na u_t\|^2_{L^2} \\&\le C\delta \|\sqrt{\n} u_{tt}\|_{L^2}^2+C(\delta)(1+\|\na u\|_{H^1}^{q_0 })\|\na u_t\|^2_{L^2}  + C(\delta)\| \na  u\|_{H^1}^4. \ea\ee
Obviously, a direct calculation gives
\be\ba |J_7|\le \delta \int \n |u_{tt} |^2dx+ C(\delta)\| \na  u\|_{H^1}^2\|\na u_t\|_{L^2}^2,\ea\ee
and by  \eqref{3d6.5},
\be\ba |J_5|&\le C\|u_t\|_{L^6}\|\na u_t\|_{L^2}\|\na\mn\|_{L^q}\|\na u\|_{L^{3q/(q-3)}}\\ &\le C \|\na u_t\|_{L^2}^2(\|\na u_t\|_{L^2}+1) .\ea\ee
For $J_6$, using integration by part and \eqref{1.1}$_1$, together with \eqref{3d6.5} and \eqref{3d.1}, we check that
\be\ba J_6&= \int (\mu(\n))_t \,\div((|D(u)|^2)_tu) dx=\int (\mu(\n))_t \,u\cdot\na(|D(u)|^2)_t dx\\
&=2\int (\mu(\n))_t \,u^k D(\partial_k u)\cdot D(u_t) dx-\int \mu(\n) \,u\cdot \na(|D(u_t)|^2) dx\\
&\quad -2\int  \,\div(u^k(\mu(\n)D(u))_t)\cdot\partial_k u_t dx\\
&=-2\int u\cdot\na\mu(\n) \,u^k D(\partial_k u)\cdot D(u_t) dx+\frac{1}{3} J_4 \\
&\quad +2\int u\cdot \na\mu(\n) \,\na u^k\cdot D(u)\cdot \partial_k u_t dx-2\int \mu(\n) \,\na u^k\cdot D(u_t)\cdot \partial_k u_t dx\\
&\quad -2\int  u^k\,\div(\mu(\n)D(u))_t\cdot\partial_k u_t dx\\
&\leq C\|u\|_{L^\infty}\|u\|_{L^6}\|\na u_t\|_{L^2}\|\na\mn\|_{L^q}\|\na u\|_{L^{3q/(q-3)}}+\delta \|\sqrt{\n} u_{tt}\|_{L^2}^2\\
&\quad+C(\delta)(1+\|\na u\|_{H^1}^{q_0 })\|\na u_t\|^2_{L^2}  + C(\delta)\| \na  u\|_{H^1}^4\\
&\quad + C\| u\|_{L^{3q/(q-3)}}\|\na u\|_{L^6}^2\|\na u_t\|_{L^2}\|\na\mn\|_{L^q}+C \|\na u\|_{H^1\cap W^{1,p_0}}\|\na u_t\|_{L^2}^2\\
&\leq\delta\|\sqrt{\n} u_{tt}\|_{L^2}^2+C(\delta)(1+\|\na u_t\|_{L^2}+\|\na u\|_{H^1}^{q_0 })\|\na u_t\|^2_{L^2}  + C(\delta)\| \na  u\|_{H^1}^4\\.\ea\ee
where in the second inequality we have use the fact that
\bnn\ba
&-2\int  u^k\,\div(\mu(\n)D(u))_t\cdot\partial_k u_t dx\\&=-2\int u^k\rho_t u_t\cdot\partial_k u_t dx-2\int u^k\rho u_{tt}\cdot\partial_k u_t dx-2\int u^k\rho_t u\cdot\na u\cdot\partial_k u_t dx\\
&\quad -2\int u^k\rho u_t\cdot\na u\cdot\partial_k u_t dx-2\int u^k\rho u\cdot\na u_t\cdot\partial_k u_t dx+2\int P_t \na u :\na u dx\\
&\leq \frac{\delta}{2} \|\sqrt{\n} u_{tt}\|_{L^2}^2+C(\delta)(1+ \|\na u \|^2_{H^1} )\|\na u_t\|_{L^2}^2+C\|\na u \|^4_{H^1},
\ea\enn
due to \eqref{1.1}$_2$, \eqref{3d6.5} and \eqref{3d.1}.

Together with these estimates of $J_i$, choosing $\delta$ suitably small, we derive from \eqref{utt1} that
\be\la{ij20}\ba & \frac{d}{dt}\int \mu(\n)|D(u_t)|^2dx +\frac12\int \n |u_{tt}|^2dx \\ &\le \theta' (t)+C  (1+\| \na u_t\|_{L^2}+\| \na u\|_{H^1}^{q_0}) \|\na u_t\|_{L^2}^2+C \|\na u\|_{H^1}^4,\ea\ee
where $\theta(t)$ satisfies
\be\la{ij22} \ba |\theta(t)|\le& C\|\sqrt{\n}u_t\|_{L^2}\|\na u_t\|_{L^2}\|\na u\|_{H^1}+C\|\n_t\|_{L^2}\| u\|_{L^6}\|u_t\|_{L^6}\|\na u\|_{L^6}\\&+C \|\na\mn\|_{L^q}\|\na u_t\|_{L^2}\|\na u\|_{H^1}^2\\ \le& \frac{1}{4}\xmu\|\na u_t\|_{L^2}^2+C\|\sqrt{\n}u_t\|_{L^2}^2 \|\na u\|_{H^1}^2+C\|\na u\|_{H^1}^4,\ea\ee due to \eqref{3a1} and \eqref{ij6}.

On the other hand, observing  that \eqref{3d6.7} gives
\bnn \la{ij23}\ba \sigma^{q_0}(1+\| \na u_t\|_{L^2}+\| \na u\|_{H^1}^{q_0}) \|\na u_t\|_{L^2}^2 \le C  \sigma^{q_0-1} \|\na u_t\|_{L^2}^4+C\sigma \| \na u_t\|^2_{L^2},\ea\enn
Then,  multiplying \eqref{ij20} by $\sigma^{q_0} e^{\lambda t} $, by Gr\"onwall's inequality, along with Lemma \ref{lem-ed1} and \eqref{ij22}, we get
\be \la{ij24}\sup_{0\le t\le T}\sigma^{q_0}e^{\lambda t}\|\na u_t\|_{L^2}^2+\int_0^T \sigma^{q_0} e^{\lambda t}\|\n^{1/2}u_{tt}\|_{L^2}^2dt\le C.\ee
Furthermore, noticing that by \eqref{ij6},
 \bnn \ba\|(\n u_t)_t\|_{L^2}^2\le C\|\na u\|_{H^1}\|\na u_t\|_{L^2
\cap L^{p_0}}^2+C\|\n^{1/2}u_{tt}\|_{L^2}^2,\ea \enn
together with \eqref{ij24}, \eqref{3d6.5}, \eqref{3d.1},  and Lemma \ref{lem-ed1},  we derive \eqref{3dr6.1} and thus completes the proof of Lemma \ref{lem5.a3}.
\end{proof}

\section{Proofs of Theorems \ref{thm1} and \ref{thm2}}

With all the a priori estimates in Section 3 at hand, we are now able to  prove Theorems \ref{thm1} and \ref{thm2}.

\emph{\bf Proof of Theorem \ref{thm1}.} According to Lemma 2.1, there exists a $T_{*}>0$ such that the system \eqref{1.1}--\eqref{inf1} has a unique local strong solution $(\rho,u,P)$ in $\Omega\times(0,T_{*}]$. By \eqref{2.2}, there exists a $T_1\in (0, T_*]$ such that \eqref{3a1} holds for $T=T_1$.

Set
\begin{equation}\label{20.1}
T^{*}\triangleq\sup \{T\, |\, \eqref{3a1}\ \text{holds}\}.
\end{equation}
It is clear that $T^*\ge T_1>0$. Hence,  it follow from   Lemma \ref{lem-ed1} and \eqref{3dr6.1}  that for any $0< \tau<T\leq T^{*}$ with $T$ finite,
\begin{equation}\label{20.2}
 \nabla u, \  P\in C([\tau,T];L^2)\cap C(\overline \Omega\times [\tau,T] ) ,
\end{equation}
where we have utilized the following standard embedding
\begin{equation*}
L^{\infty}(\tau,T;H^1\cap W^{1,p_0})\cap H^1(\tau,T;L^2)\hookrightarrow C([\tau,T];L^2)\cap C(\overline \Omega\times [\tau,T] ) .
\end{equation*}
Moreover, we deduce from \eqref{3a1}, \eqref{gj0}, \eqref{3d6.21}, and \cite[Lemma 2.3]{L1996} that
\begin{equation}\label{20.3}
\rho\in C([0,T];L^{3/2}\cap H^1), \quad\na\mn \in C([0,T];L^q).
\end{equation}
Due to \eqref{ed2} and \eqref{3dr6.1},
a standard arguments leads to
\bnn\label{20.4}
\rho   u_t \in H^1(\tau,T;L^2)\hookrightarrow C([\tau,T];L^2),
\enn
which together with \eqref{20.2} and \eqref{20.3} implies
\begin{equation}\label{20a4}
\rho   u_t+\n u\cdot\na u \in  C([\tau,T];L^2).
\end{equation}
On the other hand, noting $(\n,u)$ satisfies \eqref{3rd1} with $F\equiv \n u_t+\n u\cdot\na u$, and then by  \eqref{1.1}, \eqref{20.2},  \eqref{20.3},   \eqref{20a4},  and \eqref{3dr6.1}, we find that for any $p\in [2,p_0)$,
\begin{equation}\label{n20.4}
 \na u,~P\in C([\tau,T];D^1\cap D^{1,p}).
\end{equation}

Now, we claim that
\begin{equation}\label{20.5}
T^*=\infty.\end{equation}
Otherwise, $T^*<\infty$. By Proposition \ref{pr1}, it indicates that \eqref{3a2} holds at $T=T^*$.  By virtue of \eqref{20.2}, \eqref{20.3},  and \eqref{n20.4}, one could set
$$(\n^*,u^*)(x)\triangleq(\n, u)(x,T^*)=\lim_{t\rightarrow T^*}(\n, u)(x,t),$$
and then \bnn \n^*\in L^{3/2}\cap H^1,\quad u^*\in D^{1}_{0,\div}\cap D^{1,p}\enn  for any $p\in [2,p_0).$   Consequently,  $(\n^*,\n^*u^*)$ could be taken as the initial data and  Lemma
\ref{local} implies that there exists some $T^{**}>T^*$ such that \eqref{3a1} holds for $T=T^{**}$, which contradicts the definition of $ T^*.$
So \eqref{20.5} holds. \eqref{oiq1} and  \eqref{e} come directly from \eqref{3d6.21} and  \eqref{3dr6.1}.We complete the proof of Theorem.
 \hfill $\Box$

\emph{\bf Proof of Theorem \ref{thm2}}: We can modify slightly the proofs of Lemma \ref{lem-ed1} and \eqref{3d6.21} to obtain  \eqref{oiq2} and \eqref{eq}, here we leave it to the reader.
 \hfill$\Box$




\end{document}